\documentclass[a4paper,12pt]{amsart}
\usepackage{amsmath,amssymb}
\usepackage[shortlabels]{enumitem}
\usepackage{xcolor}
\usepackage{graphicx}
\graphicspath{ {./images/} }

\def\R{\mathbb R}
\def\Z{\mathbb Z}

\def\bcases{\begin{cases}}

\def\ecases{\end{cases}}

\newtheorem{thm}{Theorem}
\newtheorem{prop}[thm]{Proposition}

\newtheorem{lem}[thm]{Lemma}
\newtheorem{definition}[thm]{Definition}

\newtheorem{ex}[thm]{Example}

\newcommand{\eps}{\varepsilon}

\newenvironment{proof*}{\vskip 2mm\noindent {}}{\hfill $\Box$ \vskip 2mm}

\title{How regular is the evolute of a plane curve?}

\author{Nikolai Nikolov}
\address{N. Nikolov\\
Institute of Mathematics and Informatics\\
Bulgarian Academy of Sciences\\
Acad. G. Bonchev 8, 1113 Sofia, Bulgaria
\vspace{1mm}
\newline Faculty of Information Sciences\\
State University of Library Studies and Information Technologies\\
Shipchenski prohod 69A, 1574 Sofia,
Bulgaria}

\email{nik@math.bas.bg}

\author{Pascal J. Thomas}
\address{P.J. Thomas\\
Institut de Math\'ematiques de Toulouse; UMR5219 \\
Universit\'e de Toulouse; CNRS \\
UPS, F-31062 Toulouse Cedex 9, France}

\email{pascal.thomas@math.univ-toulouse.fr}

\thanks{The first named author was partially supported by the Bulgarian National
Science Fund, Ministry of Education and Science of Bulgaria under contract KP-06-N82/6. The
second named author wishes to thank the Institute of Mathematics and Informatics of the Bulgarian Academy
of Sciences for its hospitality during the time when most of this work was carried out.}

\begin{document}



\begin{abstract}
We study the relationship between the smoothness of a plane curve and that of its evolute,
especially in the cases where the parent curve is no more two or three times continuously
differentiable, and exhibit the same kind of apparent improvement in regularity: 
in the generic local situation, the evolute has one order of regularity less than the parent curve.
\end{abstract}

\maketitle

\section{Introduction}

\subsection{Motivations and outline.}

The notions of osculating circle and center of curvature of a smooth enough plane curve are classical, going back at
least to Leibnitz.  The locus of all the centers of curvature of a given curve is called its \emph{evolute} and
turns out to be the geometric envelope of the family of normal lines to the curve.  The inverse operation--finding an \emph{involute}, a
curve which the given curve is the evolute of--has 
even  found some applications to the design of some mechanisms, see for instance \cite{wiki2}.
This has other beautiful geometrice properties, some less well-known, such as the Tait-Kneser Theorem \cite{Ta}: if
the radius of curvature varies in a monotone way, then the osculating circles are nested.  Observe that the converse of this is immediate.
Most of this can be found in elementary treaties about differential geometry or advanced calculus;
a good succinct contemporary account can be found in \cite{GTT}, which singles out the Tait-Kneser Theorem as worthy of more attention.

However, the examples for which evolutes have been studied are mostly infinitely differentiable (cycloids), 
or even algebraic curves (conic sections),
see the Wikipedia page \cite{wiki1} or e.g. \cite{PRS}.  When the curves 
are bounded, their curvature must admit at least four local extrema (Four Vertex Theorem \cite{Mu}), which generate cusps
in the evolute; the interesting generalization given in \cite{FuTa} (and subsequent papers) studies such singularities,
but always in the context of the parent curve being infinitely differentiable.  

In contrast to this,  we are interested in the least possible amount of smoothness 
to be required from  the parent curve in order to get a certain smoothness for the evolute, 
and conversely.  At a first glance, it would seem that since the curvature depends on second
derivatives and the evolute is expressed in terms of it, it should have two orders 
of regularity less than the parent curve.
The result however is that the evolute has only one order of regularity 
less than the parent curve, provided
that the curvature of that curve has non-vanishing derivative, see Theorem \ref{main}. 
This ``bootstrap'' phenomenon, an apparent regularity gain due to geometric constraints, reminds us of
what happens when looking at distance functions and was pointed out in \cite{KP}, 
who attribute the initial observation to the appendix of \cite{GT}, see  \cite{NT}
for recent refinements.

Our results are local, since globally an evolute can have double points; we provide an example to that effect.

Much of the present paper may be known, but not always easy to find in current references, and we believe
we have a relatively concise exposition of the results, and proofs which do not require many prerequisites.

\subsection{Definitions and notations.}

A plane curve is a continuous map $\gamma$ from an interval $(a;b)$ to the plane $\R^2$. For $k\in \Z_+$,
$\alpha \in (0,1)$, we say the curve is of class $\mathcal C^k$ 
(resp. $\mathcal C^{k,\alpha}$) if its components are.  In order to unify notation, we will write $\mathcal C^{k,0}$ for $\mathcal C^{k}$.
Recall that a function is of class $\mathcal C^\infty$ if it is of class $\mathcal C^{k}$ for any $k$.
We routinely use the fact that if $f,g \in \mathcal C^{k,\alpha}$, with $k\ge 1$, then $f\circ g\in \mathcal C^{k,\alpha}$
(this of course fails for $k=0, \alpha<1$). 

For $k\ge 1$, we say that the curve is $\mathcal  C^{k,\alpha}$-\emph{regular}
if in addition, for all $t\in (a;b)$, $\gamma'(t)\neq (0,0)$. In this case, there exists an open interval $I$ containing $t$
such that $\gamma|_I$ embeds $I$ into $\R^2$ as a one-dimensional submanifold of the same regularity as $\gamma$.  Note
that this does not imply that $\gamma$ is globally one-to-one, only over a small enough interval.

For a  curve of class $\mathcal C^1$ and $t_0\in (a;b)$, the \emph{arc length} is $s_\gamma (t) := \int_{t_0}^t |\gamma'(u)| du$.
Changing the point $t_0$ changes $s_\gamma(t)$ by an additive constant, and we will have to do that on occasion.
When $\gamma$ is, in addition, regular, $s_\gamma$ is an increasing diffeomorphism from $(a;b)$ to $(s_\gamma (a),s_\gamma (b))$,
of at least  the same smoothness class as $\gamma$; the arclength parametrization therefore provides
the best possible smoothness for the curve among the regular parametrizations.  
We will consider $\gamma \circ s_\gamma^{-1}$ as the ``same curve parametrized
by arc length'', but will need to distinguish it in notation.  

When $\gamma$ is parametrized by arc length, $|\gamma'(s)|=1$ for any $s\in (a;b)$. There is a lifting to $\R$ 
of the map $\gamma': (a;b) \longrightarrow S^1$, so we denote $\gamma'(s)=(\cos \varphi(s), \sin \varphi(s))$
where $\varphi \in \mathcal C^{k-1,\alpha}(a;b)$  when $\gamma$ is of class
$\mathcal C^{k,\alpha}$.  Let $\nu(s):= (-\sin \varphi(s), \cos \varphi(s))$ (unit normal vector) 
be $\gamma'(s)$ rotated by $+\frac\pi2$. 

When $k\ge 2$, $\kappa(s):= \varphi'(s)$ is called the \emph{(signed) curvature}
of $\gamma$. Reducing the interval if needed, we make the standing assumption that
for $s\in(a;b)$, $\kappa(s)\neq 0$. Then we can define  $R(s):= 1/ \kappa(s)$ is the \emph{(signed) radius of curvature} of $\gamma$. 
The well-known
(plane) Frenet-Serret formulae immediately follow:
\[
\gamma''(s) = \frac1{R(s)} \nu(s), \qquad \nu'(s) = -\frac1{R(s)} \gamma'(s).
\]
The \emph{evolute} of the curve $\gamma$ is given by the map
\[
\tilde \gamma (s) := \gamma (s) + R(s) \nu (s).
\]
When we parametrize the evolute by its own arc length $\tilde s$, we write $\tilde s \mapsto \tilde \gamma_1 (\tilde s) = \tilde \gamma (s)$.

\subsection{Main results.}

\begin{thm}
\label{main}
Let $\gamma: (a;b)\longrightarrow \R^2$ be a $\mathcal C^2$-regular curve such that $\kappa(t)\neq 0$ for all $t\in(a;b)$,
and denote its evolute by $\tilde \gamma$. Then
\begin{enumerate}[(i)]
\item
$\tilde \gamma$ is $\mathcal C^1$-regular if and only if $R$ (or equivalently $\kappa$) is strictly monotone on $(a;b)$.
\item
For $k\ge 3$, $0\le \alpha \le 1$: $\gamma$ is $\mathcal C^{k,\alpha}$-regular  and $R'(s)\neq 0$, if and only if
$\tilde \gamma$ is  $\mathcal C^{k-1,\alpha}$-regular.
%\item
%For $k\ge 3$, if $\tilde \gamma$ is  $\mathcal C^{k-1}$-regular, then $\gamma$ is $\mathcal C^{k}$-regular, and its radius of curvature is monotone with a nonzero derivative. 
\end{enumerate}
\end{thm}

Note that nonvanishing of the derivative of the radius of curvature implies its monotonicity,
since we are working on an interval.

The fact that monotonicity of the radius of curvature is 
necessary and sufficient for smoothness of the evolute is well known,
and mentioned for instance in \cite[p. 308]{Co}.

\section{Proofs}

\begin{lem}
\label{increvol}
Let $\gamma$ be a $\mathcal C^2$-regular plane curve parametrized 
over $(a,b)$ by arc length $s$, and $a<s_1<s_2<b$. Let  $\tilde \gamma$ be its evolute, parametrized
by the same $s\in (a,b)$.  Then
\begin{equation}
\label{diffevol}
\tilde \gamma (s_2) - \tilde \gamma (s_1) =
\int_{s_1}^{s_2} \frac{R(s)-R(s_2)}{R(s)} \gamma'(s) ds + (R(s_2)-R(s_1)) \nu (s_1).
\end{equation}
\end{lem}

\begin{proof}
Observe first that the Frenet-Serret formulae (and the continuity of $R$) imply
\[
\nu (s_2)-\nu(s_1)= - \int_{s_1}^{s_2} \frac1{R(s)} \gamma'(s) ds = -\frac1{R(s_2)} (\gamma (s_2)-\gamma(s_1))
+ \int_{s_1}^{s_2} \left( \frac1{R(s_2)} - \frac1{R(s)} \right) \gamma'(s) ds.
\]
Then notice that
\[
\tilde \gamma (s_2) - \tilde \gamma (s_1) = \gamma (s_2)-\gamma(s_1) + R(s_2) (\nu (s_2)-\nu(s_1)) + (R (s_2)-R(s_1)) \nu (s_1).
\]
\end{proof}

\begin{proof*}{\it Proof of Theorem \ref{main}, (i).}

Suppose that $R$ is strictly monotone on $(a;b)$, we may assume that it is increasing.  Then, since $|\gamma'(s)| =1$ and 
$0\le R (s_2)-R(s) \le R (s_2)-R(s_1)$ for $s_1\le s \le s_2$,
\[
\left| \int_{s_1}^{s_2} \frac{R(s)-R(s_2)}{R(s)} \gamma'(s) ds \right| \le |R(s_2)-R(s_1)| \int_{s_1}^{s_2} \frac{1}{|R(s)|}  ds
\le  |s_2-s_1| \frac{|R(s_2)-R(s_1)|}{\min_{[s_1,s_2]} |R|}.
\]
So  \eqref{diffevol} implies
\begin{equation}
\label{tangevol}
\frac{\tilde \gamma (s_2) - \tilde \gamma (s_1)}{R(s_2)-R(s_1)} = \nu(s_1) + O(|s_2-s_1|).
\end{equation}
This implies that the curve admits a tangent vector, parallel to the unit vector $\nu(s)$, and that 
$s\mapsto R(s)$, which is a homeomorphism of intervals (by strict monotonicity) provides a (continuous) parametrization  
by arclength.  Let $\tilde s:= R(s)$, $\tilde \gamma_1 (\tilde s) = \tilde \gamma (R^{-1}(\tilde s))$ 
and $\tilde \gamma_1' (\tilde s) = \nu (R^{-1}(\tilde s))$, which is a continuous non-vanishing vector function of $\tilde s$,
so $\tilde \gamma_1$, and therefore $\tilde \gamma$, is $\mathcal C^1$-regular.

Conversely, assume that near a point $s_1$, $R$ is not strictly monotone.  

Exchanging the roles of $s_1$ and $s_2$ in \eqref{diffevol}, we have 
\[
\tilde \gamma (s_2) - \tilde \gamma (s_1) =
\int_{s_1}^{s_2} \frac{R(s)-R(s_1)}{R(s)} \gamma'(s) ds + (R(s_2)-R(s_1)) \nu (s_2).
\]

First notice that we can find  values of $s_2 \ge s_1$ arbitrarily close to $s_1$
 such that $|R(s_2)-R(s_1)|=\max_{s\in[s_1;s_2]} |R(s)-R(s_1)|$, and likewise for $s_2\le s_1$. 
The  argument above, applied to sequences of such values $s_{2,n}\to s_1$,
 leads to
\[
 \frac{\tilde \gamma (s_{2,n}) - \tilde \gamma (s_1)}{R(s_{2,n})-R(s_1)} = \nu(s_{2,n}) + O(|s_{2,n}-s_1|),
\]
and since $\nu(s_{2,n})\to \nu(s_1)$,
in all cases $\nu(s_1)$ belongs to the tangent cone
of $\tilde \gamma$ at $\tilde \gamma (s_1)$. 

There are three cases.

{\bf Case 1.}

If $R$ admits a strict local extremum at $s_1$, assume without loss of generality that it is a local maximum.
With $\langle \cdot , \cdot \rangle$ standing for the Euclidean
inner product in the plane,
\begin{equation}
\label{DL1}
\langle \tilde \gamma (s_2) - \tilde \gamma (s_1), \nu(s_1) \rangle =
(R (s_2)-R(s_1)) + \int_{s_1}^{s_2} \frac{R(s)-R(s_2)}{R(s)} \langle \gamma'(s), \nu(s_1) \rangle ds,
\end{equation}
so for any $\eps>0$, if
$s_2=s_1+h$ with $h$ chosen as above, $|R(s)-R(s_2)| \le |R(s_1)-R(s_2)|$ for $s\in [s_1;s_2]$
and $\langle \gamma'(s), \nu(s_1) \rangle = \langle (\gamma'(s)-\gamma'(s_1)), \nu(s_1) \rangle= O(|s-s_1|)$, so
the integral term is bounded in modulus by $C|s_1-s_2|^2 |R(s_1)-R(s_2)|$ and 
$\langle \tilde \gamma (s_2) - \tilde \gamma (s_1), \nu(s_1) \rangle >0$.  For values $h<0$  close enough to $0$, 
the same argument again yields  $\langle \tilde \gamma (s_2) - \tilde \gamma (s_1), \nu(s_1) \rangle >0$,
so $\tilde \gamma$ cannot be regular at $s_1$.

{\bf Case 2.}

If $R$ admits a non-strict local extremum at $s_1$, there must be values of $h$ arbitrarily close to $0$
such that $R(s_1+h)=R(s_1)$. So if $\tilde \gamma$ was differentiable at $s_1$, we would have $\tilde \gamma'(s_1)=0$,
and this would be true for any (monotone) reparametrization of the curve, so it cannot be regular.

{\bf Case 3.}

Finally if $R$ does not admits a strict local extremum at $s_1$, and is not strictly monotone in any neighborhood of $s_1$,
there must be a sequence of values $h_n\to0$ such that $R$ admits a local extremum at $s_1+h_n$, so 
$\tilde \gamma$ cannot be $\mathcal C^{1}$-regular at $s_1$.

\vskip.4cm

{\it Proof of Theorem \ref{main}, (ii).}

Direct part:

Since $R(s)= \left( \langle \gamma''(s),  \frac{\gamma'(s)}{|\gamma'(s)|} \rangle \right)^{-1}$, it is immediate
that $R$ and $\tilde \gamma$ are of class $\mathcal C^{k-2,\alpha}$. The hypothesis about the derivative of $R$ means
that $R$ is a $\mathcal C^{k-2,\alpha}$ diffeomorphism onto its image interval. When we reparametrize $\tilde \gamma$
by $R$, $\tilde \gamma_1' (\tilde s) = \nu (R^{-1}(\tilde s))$. 
The vector $\nu$ is obtained by a rotation of $+\pi/2$ from $\gamma'(s)$ and so 
$s\mapsto \nu(s)$ is of class $\mathcal C^{k-1,\alpha}$.
Since $R^{-1}$ is of class  $\mathcal C^{k-2,\alpha}$ and $k-2\ge 1$, $\tilde \gamma_1'$
  is a $\mathcal C^{k-2,\alpha}$ non-vanishing vector function of $\tilde s$, so $\tilde \gamma_1$ is $\mathcal C^{k-1, \alpha}$-regular.
 
\vskip.4cm

%{\it Proof of Theorem \ref{main}, (iii).}

Converse:

Now assume that $\tilde \gamma$ is  $\mathcal C^{k-1,\alpha}$-regular.

It is a classical fact (see e.g. \cite{Co} or \cite{wiki2}) that
if we denote by  $\tilde s\mapsto \tilde \gamma_1(\tilde s)$ the arclength parametrization of the evolute, then there exists some $c\in \R$,
necessarily outside of the source interval of that parametrization since we always assume that the original curve
 $\gamma$ is $\mathcal C^2$-regular,
such that a parametrization of $\gamma$ is given by 
\begin{equation}
\label{invol}
\gamma_1(\tilde s)= \tilde \gamma_1 (\tilde s)+ (c-\tilde s) \tilde \gamma_1'(\tilde s),
\end{equation}
which shows that under our hypothesis $\gamma_1$ is of class $\mathcal C^{k-2, \alpha}$ at least, with $k-2\ge 1$,
and so is the arclength parametrization of $\gamma_1$, %is of class $\mathcal C^{k-2}$ at least.
$\tilde s\mapsto s=R^{-1}(\tilde s)$.

We need to show that $\gamma_1'(\tilde s)= (c-\tilde s) \tilde \gamma_1''(\tilde s)$ never vanishes on the relevant interval.  
Since $\tilde s\neq c$ for all the values we consider, it is enough to show that $\tilde \gamma_1''(\tilde s)$,
which is well-defined since $k-1\ge2$, never vanishes.  

But Lemma \ref{increvol} implies that $\tilde \gamma_1'(\tilde s)=\nu (R^{-1}(\tilde s))$.  
The derivative  $\frac{d\nu}{ds}$ never vanishes by the hypothesis
that the curvature of the initial $\gamma$ is nonzero, and $R^{-1}$ is the inverse of a $\mathcal C^1$ function,
so $|R^{-1}(\tilde s_1)-R^{-1}(\tilde s_2)|\ge C|s_1-s_2|$. If we had $ (\nu \circ R^{-1})'(\tilde s_1)=0$,
then we would have 
\[
| \nu(R^{-1}(\tilde s_2))-\nu(R^{-1}(\tilde s_1))| = o (|s_1-s_2|) = o(|R^{-1}(\tilde s_1)-R^{-1}(\tilde s_2)|),
\]
so $\nu'(R^{-1}(\tilde s_1))=0$, a contradiction.

Finally $\nu(s)= \tilde \gamma_1'(R( s))$ must be of class $\mathcal C^{k-2,\alpha}$.
Since $\gamma'(s)$ is deduced from $\nu(s)$ by a rotation of angle $-\frac\pi2$, $\gamma'$ is of class $\mathcal C^{k-2,\alpha}$.
Now the Frenet-Serret formulae give $\nu'(s) = -\frac1{R(s)} \gamma'(s)$, so $\nu'$ is of class $\mathcal C^{k-2,\alpha}$
as a product of such. This implies that $\nu$ is of class $\mathcal C^{k-1,\alpha}$,
so $\gamma'$, deduced from it by a rotation of angle $-\pi/2$, has the same smoothness, and $\gamma$ itself is of class $\mathcal C^{k,\alpha}$.

The fact that the derivative of $R$ does not vanish is equivalent to $\kappa'\neq 0$, where $\kappa=1/R$ is the curvature of $\gamma$.
The curvature of $\tilde \gamma$ is given by $-\kappa^3/\kappa'$; since we have assumed $\kappa(s)\neq0$ for all $s$, 
and $\tilde\gamma$ is of class at least $\mathcal C^2$, the curvature of $\tilde \gamma$ is always bounded and $\kappa'$,
and thus $R'$, do not vanish.

The fact that $R'(s)\neq0$ will also follow from the next section, which states more
precise results about the behavior of the evolute in a neighborhood of a point $s_1$ where $\gamma$ is smooth and $R'(s_1)=0$.
\end{proof*}

\section{Local behavior}

In this section, we will investigate the situations where $R'(s)\neq 0$ no longer holds.
We need to quantify the vanishing of the derivative of the (radius of) curvature.

\begin{definition}
\label{order}
Let $I$ be an open interval, $f:I\longrightarrow R$ be a continuous function, and $s_1\in I$. We set
\begin{align*}
m_f(s_1)^-:= &\sup \{ \mu\ge0: \limsup_{s\to s_1} \frac{|f(s)-f(s_1)|}{|s-s_1|^\mu} <\infty \} \in [0;\infty],
\\
m_f(s_1)^+:= &\inf \{ \mu\ge0: \liminf_{s\to s_1} \frac{|f(s)-f(s_1)|}{|s-s_1|^\mu} >0 \} \in [0;\infty].
\end{align*}
\end{definition}
Note that $m_f(s_1)^+\ge m_f(s_1)^-$.
When $f$ is of class $\mathcal C^k$, with $k\ge m_f(s_1)$, and if there exists $l\in \{1,\dots,k\}$
such that $f^{(l)}(s_1)\neq0$,
\begin{multline*}
m_f(s_1)=m_f(s_1)^+=m_f(s_1)^-
\\
=\max\{ m\ge 1: f^{(l)}(s_1)=0, 1\le l \le m\} + 1= \min\{m\ge 1:f^{(m)}(s_1)\neq 0\}.
\end{multline*}
Notice that when this happens and when $f$ is strictly monotone on a neighborhood of $s_1$, $m_f(s_1)$ must be an odd integer.

With the notations from the introduction, notably a curve $\gamma$ parametrized by arclength $s$,
near a point $s_1$, the \emph{local frame} for the evolute $\tilde \gamma$
will be given by the point $\tilde \gamma (s_1)$ (which we will take to be the origin), 
and the direct orthonormal basis $(\nu(s_1),-\gamma'(s_1))$. 
By a point $(x,y)$ we mean $\tilde \gamma (s_1)+x\nu(s_1)-y\gamma'(s_1)$.

\begin{prop}
\label{locbeh}
Let $\gamma: (a;b)\longrightarrow \R^2$ be a $\mathcal C^2$-regular curve, denote its evolute by $\tilde \gamma$. Let $s_1\in (a;b)$.
\begin{enumerate}[(i)]
\item
Suppose that the signed radius of curvature
$R$ of $\gamma$ is strictly  increasing on $(a;b)$ and 
let $\omega(x):=R(s_1+x)-R(s_1)$.
Then $\tilde \gamma$
can be represented in the local frame by $y=g(x)$, where $g(x)= O(x \omega^{-1}(x))$.
\item
If the evolute parametrized by its arclength, $\tilde \gamma_1$, is $\mathcal C^{1,\alpha}$-regular  near $s_1$ for some $\alpha>0$,
then $m_R(s_1)^- \le 1/\alpha$.

\noindent
If the derivative of $R$ vanishes to infinite order at $s_1$, the curve $\gamma$ is of class $\mathcal C^{1,0}$,
and no better in the $\mathcal C^{1,\alpha}$ scale.
\item
In particular, if $\gamma$ is $\mathcal C^k$-regular,  if $R$  is strictly  increasing on $(a;b)$ 
and $1<m:=m_R(s_1) \le k$,  $\tilde \gamma$
is $\mathcal C^{1,1/m}$-regular  near $s_1$ with $m\ge 3$, and cannot be any smoother.
\item
If $\gamma$ is $\mathcal C^k$-regular and $m:=m_R(s_1)$ is a nonzero even integer, then $R$
is not monotone in any neighborhood of $s_1$ and $\tilde \gamma$ admits a cusp: it is represented,
changing $\nu$ into $-\nu$ if needed,
by the union of two graphs of the form $y=g_1(x)$ and $y=-g_2(x)$, $x\ge 0$, with $g_1(x), g_2(x) \ge 0$ 
and $g_1(x), g_2(x) \asymp x^{1+1/m}$.
\end{enumerate}
\end{prop}

%Notice that when $R$ is strictly monotone, or even $R(s)\neq R(s_1)$ in a one-sided neighborhood of $s_1$, we can find some strictly increasing $\omega$ satisfying the hypotheses.
%When $0<m_f(s_1)^+<\infty$, we can take $\omega(x)= x^{m_f(s_1)^+}$.  

Notice that in item (i) we could have used any  strictly increasing $\omega$ such that for $s$ in a neighborhood of $s_1$, $|R(s)-R(s_1)| \ge \omega(|s-s_1|)$. This can apply to cases where $R$ is not monotone.

\begin{proof}

{\it (i).}

From \eqref{DL1} and the subsequent considerations we see that
\begin{multline}
\label{DL2}
\langle \tilde \gamma (s_2) - \tilde \gamma (s_1), \nu(s_1) \rangle =
(R (s_2)-R(s_1)) + \int_{s_1}^{s_2} \frac{R(s)-R(s_2)}{R(s)} \langle \gamma'(s)-\gamma'(s_1), \nu(s_1) \rangle ds
\\
= (R (s_2)-R(s_1)) (1 + O(|s_2-s_1|^2)), 
\\
\langle \tilde \gamma (s_2) - \tilde \gamma (s_1), \gamma'(s_1) \rangle = 
 \int_{s_1}^{s_2} \frac{R(s)-R(s_2)}{R(s)} \langle \gamma'(s),\gamma'(s_1) \rangle ds
 \\
 =  O(|R (s_2)-R(s_1)||s_2-s_1|).
\end{multline}
This implies in particular that the projection from the image of $\tilde \gamma$ to the $x$-axis is one-to-one,
therefore we can represent that image curve as a graph $y=g(x)$; and since $x\asymp R (s_2)-R(s_1)$,
$R (s_2)-R(s_1) \ge \omega(s_2-s_1)$ for $s_2 >s_1$, $R (s_2)-R(s_1) \le -\omega(s_2-s_1)$ for $s_2 <s_1$,
and  $|y| \preceq |R (s_2)-R(s_1)||s_2-s_1|$,
we have $|y|\preceq x \omega^{-1}(x)$.

{\it (ii).}

Since $\gamma$, parametrized by arclength, is (at least) $\mathcal C^2$-regular, we can write $s=\gamma^{-1}(\gamma(s))$,
where $\gamma^{-1}$ is of class $\mathcal C^2$
and only defined locally on the embedded manifold $\gamma(I)$, $I$ a small enough interval.  However
$\gamma$ is also given as an involute of $\tilde \gamma$ by $\gamma_1(\tilde s)= \tilde \gamma_1 (\tilde s) - \tilde \gamma_1' (\tilde s) (s-c)$,
where $\tilde s$ is arclength on the evolute $\tilde \gamma$.  By the regularity hypothesis, $\gamma_1$ will be of class 
$\mathcal C^{0,\alpha}$ and $s=\gamma^{-1}(\gamma_1(\tilde s))$, so there is a constant $C>0$
such that 
\[
|s_2-s_1|\le C |\tilde s_2-\tilde s_1|^\alpha =C|R (s_2)-R(s_1)|^\alpha ; 
\]
if $m_R(s_1)^-<\infty$, for any $\eps>0$ we have 
$C_\eps $ such that $|R (s_2)-R(s_1)|\le C_\eps |s_2-s_1|^{m_R(s_1)^- -\eps}$,
so that $|s_2-s_1|\le C_\eps' |s_2-s_1|^{\alpha (m_R(s_1)^- -\eps)}$, so $\alpha m_R(s_1)^-\le 1$.  

If $m_R(s_1)^-=\infty$, $|R (s_2)-R(s_1)|\le C |s_2-s_1|^{1+1/\alpha}$, which leads to a contradiction.

For the second statement, Theorem \ref{main} implies that $R$ is of class $\mathcal C^1$.
If $R$ was of class $\mathcal C^{1,\alpha}$, $\alpha>0$, we would have $ m_R(s_1)^-\le 1/\alpha$
which contradicts the hypothesis of vanishing of the derivative to infinite order.

{\it (iii).}

The hypothesis implies that $R'$ admits an isolated zero at $s_1$, so on a punctured neighborhood of $s_1$, 
$\tilde \gamma$ is of class at least $\mathcal C^2$ since $k\ge3$.  We have $\omega(x)\asymp x^m$ and
deduce from (i) that $\tilde \gamma$ can be represented in the local frame by $y=g(x)$ et 
$g(x)=O(|x|^{1+1/m})$, in particular $\tilde \gamma$ is differentiable at $s_1$ with $\tilde \gamma_1'(\tilde s_1)=\nu(s_1)$. 

More generally, for $s,s'$  in a neighborhood of $s_1$, $\tilde \gamma'(R(s))=\nu(s)$ and $|\nu(s')-\nu(s_1)| \preceq |s'-s|$;
we claim that
\[
|s'-s| \preceq |R (s')-R(s)|^{1/m}.
\]
To see this, assume without loss of generality that $R^{(m)}(t) \ge c >0$ for $t$ in an open interval $I$ around $s_1$, and $s, s'\in I$,
$s_1\le s < s'$. 
An easy induction argument then shows that for $0 \le j \le m-1$, $R^{(m-j)}(s) \ge \frac{c}{j!} (s-s_1)^j \ge 0$. By the
Taylor-Lagrange formula with integral remainder,
\begin{multline*}
R (s')-R(s) = \sum_{j=1}^{m-1} \frac{(s'-s)^j}{j!} R^{(j)}(s) + \int_s^{s'} \frac{(s'-t)^{m-1}}{(m-1)!} R^{(m)}(t) dt
\\
\ge c \int_s^{s'} \frac{(s'-t)^{m-1}}{(m-1)!}  dt = c \frac{(s'-s)^{m}}{m!}.
\end{multline*}
This proves the claim.

The regularity cannot be improved as a consequence of (ii).

{\it (iv).}

Now $R$ has a strict local extremum. Suppose that it is a minimum,
and that in a neighborhood of $s_1$ we have $R(s)-R(s_1) = c (s-s_1)^m +o(|s-s_1|^m)$, with $c>0$.  Then in the local frame,
$0\le x=R(s)-R(s_1) \asymp (s-s_1)^m$, and for $s\ge s_1$, $y \asymp (s-s_1)^{m+1}$.  For $s\le s_1$, $-y \asymp -(s-s_1)^{m+1}$,
and the claim easily follows. In the case of a local maximum, we change the sign of $\nu$ to obtain the same type of cusp 
in a modified local frame.
\end{proof}

\begin{proof*} {\it End of Proof of Theorem \ref{main} (ii), converse part.}

We already know by the proof of Theorem \ref{main} (i) that $R$ must be strictly monotone. If $R'$ vanished at $s_1$, 
we would have to have $m_R(s_1)\ge 3$, so $\tilde \gamma$ is at best of class $\mathcal C^{1,1/3}$ by Proposition \ref{locbeh}(ii) and (iii),
and cannot be of class $\mathcal C^{2}$.
\end{proof*}

\section{Multiple points}

Our definition of a regular curve restricts attention to a small enough interval in order to avoid problems with double points. 
Can those actually occur?

One can see that an evolute of a regular, one-to-one curve may exhibit double points.

\begin{ex}
There exists a $\mathcal C^\infty$ regular curve $\gamma$ such that its evolute is a lima\c con $\tilde \gamma$ given, for instance, by
the polar representation $r(\theta) = 1 + 2 \cos \theta$, $-\frac{3\pi}4 < \theta < \frac{3\pi}4$. 
\end{ex}

The curves $ \gamma$ and $\tilde \gamma$ are represented in Figure 1.

\begin{proof}
We start from a parametric representation for an arc of lima\c con,
which has a double point at $(0,0)$ corresponding to the values
$\theta= -2\pi/3, 2\pi/3$ of the parameter:
\[
\tilde \gamma(\theta) = 
\begin{pmatrix} {\cos \theta}\\{\sin \theta} \end{pmatrix}
(1 + 2 \cos \theta).
\]
We deduce the arc length from $|\gamma'(\theta) |=5+4\cos\theta$.
An involute for $\tilde \gamma$ is given by
\begin{multline*}
\gamma(\theta) = 
\begin{pmatrix} {\cos \theta}\\{\sin \theta} \end{pmatrix}
\left( 1 + 2 \cos \theta +
2 (\sin \theta)\frac{\int_{-\pi}^{\theta}\sqrt{5 + 4 \cos u}\, du}{\sqrt{5 + 4 \cos \theta}} 
\right) 
\\- 
\begin{pmatrix} {-\sin \theta}\\{\cos \theta} \end{pmatrix}
(1 + 2 \cos \theta)\frac{\int_{-\pi}^{\theta}\sqrt{5 + 4 \cos u}\, du}{\sqrt{5 + 4 \cos \theta}} 
\\
=\begin{pmatrix} {\cos \theta}\\{\sin \theta} \end{pmatrix}
(1 + 2 \cos \theta)
+
\begin{pmatrix}  (\sin \theta)(1 + 4 \cos \theta) \\
{-4 \cos^2 \theta - \cos \theta + 2} \end{pmatrix}
  \frac{\int_{-\pi}^{\theta}\sqrt{5 + 4 \cos u}\, du}{\sqrt{5 + 4 \cos \theta}} .
\end{multline*}
From the first formula, we easily see that $\gamma$ is regular.
Notice that by what follows, $\gamma$ is also a simple curve.
\end{proof}

\includegraphics{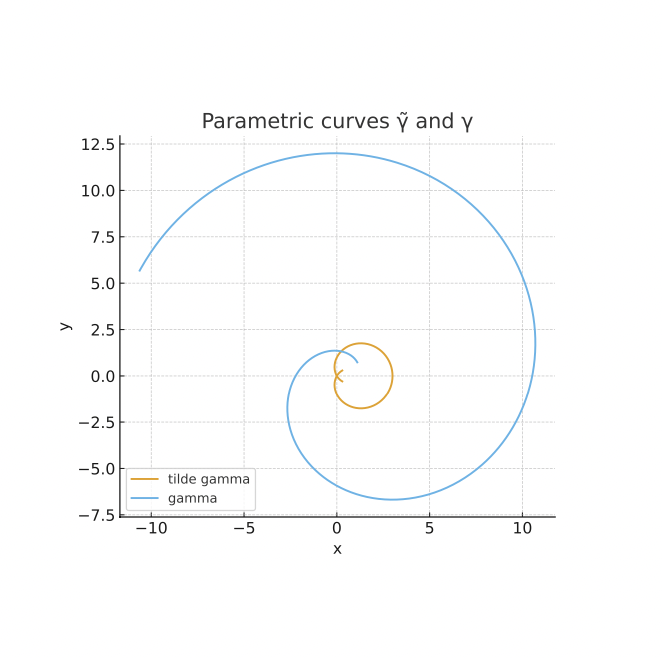}
\begin{center}
\vskip-2cm

Figure 1.
\end{center}

\begin{lem}
\label{involutesimple}
The involute of a smooth curve, when it is a regular curve, has no double point.  
\end{lem}

Recall from the proof of Theorem \ref{main} (ii) (converse) that %taking derivative of \eqref{invol}, we find
$\gamma_1'(\tilde s)= (c-\tilde s) \tilde \gamma_1''(\tilde s)$,
so even if it does not start from a point on the curve $\tilde \gamma$,
the involute could have a singularity when $\tilde \gamma$ undergoes an inflection.

On the other hand the same proof shows that if we start from a curve obtained as an evolute of a regular curve, 
the involute is regular outside of the possible points of contact with the evolute, which is not surprising
since it will be a parallel to the parent curve.

\begin{proof*}{\it Proof of Lemma \ref{involutesimple}.}

By construction, the radius of curvature of an involute, being of the form 
$c-\tilde s$ must be strictly monotone.  We conclude with the next result.
\end{proof*}

\begin{lem}
\label{monocurvsimple}
A smooth regular curve with strictly monotone curvature cannot have a double point.
\end{lem}

\begin{proof}
Let us assume that with the parametrization we have chosen, the radius of curvature is strictly increasing. We claim that
near each point of the curve, $\gamma(s)$ lies inside its osculating circle at $\gamma(s_1)$  for $s<s_1$ and outside 
of it for $s>s_1$. Indeed the Tait-Kneser Theorem \cite{Ta} tells us that osculating circles are nested within each other,
so if $s_1<s_2$, the point $\gamma(s_2)$ is on the osculating circle at $\gamma(s_2)$,
and $R(s_2)>R(s_1)$, so that circle is entirely outside the osculating circle at $\gamma(s_1)$, on which lies $\gamma(s_1)$.
This way $\gamma(s_2)$ can never become equal to a previously taken value.
\end{proof}

{}

\end{document}